\documentclass{article}
\usepackage[utf8]{inputenc}
\usepackage[a4paper, total={6in, 8in}]{geometry}

\usepackage{hyperref,amssymb,amsmath,amsthm,enumerate,enumitem,mathbbol,subcaption}
\usepackage{xcolor}
\usepackage{comment}
\usepackage[capitalise]{cleveref}
\usepackage{pgfplots}
\usepackage{mdframed}
\usepackage{placeins}
\usepackage{verbatim}

\usetikzlibrary{positioning,arrows}
\pgfplotsset{compat=1.17}

\newtheorem{thm}{Theorem}[section]
\newtheorem{lemma}[thm]{Lemma}
\newtheorem{prop}[thm]{Proposition}
\newtheorem{cor}[thm]{Corollary}

\newtheorem{defn}[thm]{Definition}
\newtheorem{notat}[thm]{Notation}

\newtheorem{exmp}[thm]{Example}

\newtheorem{remark}[thm]{Remark}

\renewcommand{\a}{\alpha}
\renewcommand{\b}{\beta}
\renewcommand{\c}{\gamma}
\renewcommand{\t}{\theta}
\renewcommand{\O}{\mathcal{O}}
\newcommand{\frakp}{\mathfrak{p}}
\newcommand{\frake}{\texttt{e}}
\newcommand{\frakf}{\texttt{f}}
\newcommand{\frakg}{\texttt{g}}
\newcommand{\FF}{\mathbb{F}}
\newcommand{\NN}{\mathbb{N}}
\newcommand{\C}{\mathbb{C}}
\newcommand{\Norm}{\mathcal{N}}
\newcommand{\QQ}{\mathbb{Q}}
\newcommand{\ZZ}{\mathbb{Z}}
\newcommand{\Z}{\mathbb{Z}}
\newcommand{\Ia}{\mathfrak{a}}
\newcommand{\Ip}{\mathfrak{p}}
\renewcommand{\k}{\mathbb{k}}

\title{First-degree prime ideals of composite extensions}
\author{Giordano Santilli$^1$ \and Daniele Taufer$^2$}
\date{ \small $^1$Università degli Studi di Trento - \href{mailto: giordano.santilli@gmail.com}{giordano.santilli@gmail.com}\\%
    $^2$KU Leuven - \href{mailto: daniele.taufer@kuleuven.be}{daniele.taufer@kuleuven.be}\\[2ex]%
    \normalsize March 2024}

\begin{document}
 
\maketitle

\begin{abstract}
    Let $\QQ(\a)$ and $\QQ(\b)$ be linearly disjoint number fields and let $\QQ(\t)$ be their compositum.
    We prove that the first-degree prime ideals of $\ZZ[\t]$ may almost always be constructed in terms of the first-degree prime ideals of $\ZZ[\a]$ and $\ZZ[\b]$, and vice-versa. We also classify the cases in which this correspondence does not hold, by providing explicit counterexamples.
    We show that for every pair of coprime integers $d,e \in \ZZ$, such a correspondence almost always respects the divisibility of principal ideals of the form $(e+d\t)\ZZ[\t]$, with a few exceptions that we characterize.
    Finally, we discuss the computational improvement of such an approach, and we verify the reduction in time needed for computing such primes for certain concrete cases.
\end{abstract}



\section{Introduction}

Let $\mathcal{O}$ be the ring of integers of a number field $\QQ(\t)$.
It is well-known that the norm of its prime ideals is always a prime power $p^e$, and this property also holds for every sub-order of $\mathcal{O}$, such as $\ZZ[\t]$.
A special family of primes that deserves particular attention is composed of those of degree $e = 1$, namely those of prime norm.
Such \emph{first-degree prime ideals} have been classically studied as they constitute a set of basic components for ideals.
In fact, a positive fraction of prime integers splits only by means of first-degree primes, and the class group of any Galois field may be generated from products of such ideals \cite{Hilbert}.

More recently, similar results have been obtained in a more applied framework: first-degree prime ideals of $\ZZ[\t]$ have been proved to constitute a basis for principal ideals generated by $e+d\t$ in $\ZZ[\t]$ for every coprime pair $e,d \in \ZZ$ \cite{BLP}, and this evidence has been exploited for designing the celebrated General Number Field Sieve (GNFS) algorithm \cite{LLMP,BL}, which is nowadays the most efficient classical algorithm known for factoring large integers.
Indeed, after a parameters selection phase, such an algorithm needs to compute large sets of first-degree prime ideals of $\ZZ[\t]$, which will be employed for factoring the aforementioned principal ideals. These factorizations will be therefore sieved in order to detect certain relations, that should lead to the factorization of the input integer with a positive probability.
Moreover, the same algorithm has been proven effective for solving the discrete logarithm problem over finite fields, both for prime \cite{GNFS4Fp} and power-of-prime \cite{GNFS4FqA,GNFS4FqB} fields.

In this paper, the theory of first-degree prime ideals of $\ZZ[\t]$ is further enhanced by establishing their relation with the corresponding prime ideals obtained from the minimal (non-trivial) sub-fields of $\QQ(\t)$.
The novelty of this work is twofold.
From a theoretical perspective, whenever $\QQ(\t)$ is realized as the compositum of two linearly disjoint sub-fields $\QQ(\a)$ and $\QQ(\b)$, the factorization of $(e+d\t)$ is proved to be almost always readable from the divisibility of its relative norm in $\ZZ[\a]$ and $\ZZ[\b]$.
On a computational side, the described procedure leads to a more efficient method for producing first-degree primes of $\ZZ[\t]$, outperforming the standard algorithm of a linear factor which depends on the smoothness of the extension degree $[\QQ(\t):\QQ]$.

More precisely, employing the convenient description of such primes \cite{BLP} as 
\[
    (t,p) = \ker (\Z[\t] \to \FF_p, \ \t \mapsto t),
\]
the \emph{combination} of first-degree primes $(r,p) \subseteq \ZZ[\a]$ and $(s,p) \subseteq \ZZ[\b]$ is defined as $(r+s,p) \subseteq \ZZ[\t]$, and such an operation is proved to describe the vast majority of such primes in $\ZZ[\t]$.
Furthermore, the divisibility of principal ideals $I = (e+d\t)\ZZ[\t]$ is respected in all but exceptional cases, which are fully characterized in terms of the zeroes of the affine map
\begin{equation*}
    \phi : \FF_p \to \FF_p, \qquad x \mapsto -x - d^{-1}e.
\end{equation*}

The main novel results of this paper are collected in Table \ref{tab:recap1}.
Its first row indicates when the combination of first-degree prime ideals in $\ZZ[\a]$ and $\ZZ[\b]$ dividing $I_\a = I \cap \ZZ[\a]$ and $I_\b=I \cap \ZZ[\b]$ is a first-degree prime ideal of $\ZZ[\a+\b]$, and when it divides $I$. 
The second row depicts the opposite scenario, namely when a first-degree prime ideal of $\ZZ[\a+\b]$ dividing $I$ determines first-degree prime ideals in $\ZZ[\a]$ and $\ZZ[\b]$, and when they divide $I_\a$ and $I_\b$.

 \begin{table}[!htb]
 \begin{center}
 \resizebox{\textwidth}{!}{
\begin{tabular}{ |c|c|c| } 
\cline{2-3}
  \multicolumn{1}{c|}{} & Existence & Divisibility \\
 \hline 
 \rule{0pt}{3\normalbaselineskip} $(r,p), (s,p) \implies (t,p)$ & Always  & unless $\begin{cases}g(\phi(r)) \equiv 0 \bmod p \\ f(\phi(s)) \equiv 0 \bmod p \\ \phi(r) \not\equiv s \bmod p \\ \phi(s) \not\equiv r \bmod p\end{cases}$ \\ 
  & (Proposition \ref{idealsdownup}) & (Theorem \ref{thm:divbeltoab}) \\[0.2cm]
 \hline
 \rule{0pt}{1.2\normalbaselineskip}   & when: &\\ $(t,p) \implies (r,p), (s,p)$  & $t$ is a simple root of $\textup{minpol}_{\QQ}(\a+\b) \bmod{p}$, or  & Always\\ 
 & $\QQ(\a)$ and $\QQ(\b)$ are normal and of coprime degrees &  \\
  & (Propositions \ref{updown1} and \ref{updown2}) & (Theorem \ref{thm:divabtobel}) \\[0.2cm]
 \hline
\end{tabular}}
\caption{Overview of the main results of the paper.}
\label{tab:recap1} 

 \end{center}
 \end{table}

Such results lead to a bottom-up approach that may be employed for speeding up the production of these primes, as well as for designing new approaches based on the smaller extensions, whose usage is often preferable.

The employed hypotheses are not truly restrictive: every pair of reasonably uncorrelated fields happen to be linearly disjoint \cite{Cohn,CoprimeDiscriminant}, thus every composite extension may be realized this way, with a suitable choice of sub-extensions.
\emph{Ad hoc} examples are provided to show that every required hypothesis is essential.
 
This paper extends a previous work of the authors \cite{SanTau}, which addresses the same problem when the field $\QQ(\t)$ is biquadratic.
However, the techniques employed and developed in the current paper are more sophisticated and lead to a deeper comprehension of ideals in towers of fields.
The novel results not only generalize those of \cite{SanTau}, but also over a much wider range of situations and provide theoretical tools that may be exploited for computational and cryptographic purposes, such as factoring and sieving through number fields.

This paper is organized as follows: in Section \ref{sec:Preliminaries} the basic results about resultant and linearly disjoint extensions are recalled and combined to properly determine the field extensions that we address in the present work.
Section \ref{sec:FDPI} is devoted to defining the first-degree prime ideals combination and to establishing when this construction defines a complete correspondence of the considered first-degree prime ideals.
Such an association is proved to almost always respect the divisibility of prescribed principal ideals in Section \ref{sec:Divisibility}.
In Section \ref{sec:computational}, the complexity of a combination-based approach for computing first-degree prime ideals is discussed, and a computational comparison with the current method is presented.
Finally, in Section \ref{sec:conclusions} we review the work and hint at possible future research directions.

 
 \section{Preliminaries} \label{sec:Preliminaries}
 
 \subsection{Resultant}
 
 In this section, we recall the main properties of the polynomial resultant over a field.
 
 \begin{defn}[Resultant]
 Let $\k$ be a field and $f = \sum_{i = 0}^n a_i x^i, \ g = \sum_{i = 0}^m b_i x^i \in \k[x]$ be polynomials of degree $n$ and $m$, i.e. $a_nb_m \neq 0$.
 The resultant $R(f,g)$ of $f$ and $g$ is defined as the determinant of their Sylvester matrix, i.e.
 \[
 R(f,g) = \det \begin{pmatrix}
 a_n & a_{n-1} & a_{n-2} & \ldots & a_{0} & 0 & 0 & \ldots & 0 \\
 0 & a_n & a_{n-1} & \ldots & a_1 & a_{0} & 0 & \ldots & 0 \\
 \vdots & \vdots & \vdots & & \vdots & \vdots & \vdots & & \vdots \\ 
  0  & \ldots & 0 & a_{n} &  & \ldots &  & a_1 & a_{0} \\
  b_m & b_{m-1} & b_{m-2} & \ldots & b_{0} & 0 & 0 & \ldots & 0 \\
 0 & b_m & b_{m-1} & \ldots & b_1 & b_{0} & 0 & \ldots & 0 \\
 \vdots & \vdots & \vdots & & \vdots & \vdots & \vdots & & \vdots \\ 
  0  & \ldots & 0 & b_{m} &  & \ldots &  & b_1 & b_{0} 
 \end{pmatrix}.
\]
 \end{defn}
Hence, the resultant is the determinant of a $(n+m) \times (n+m)$ matrix, whose first $m$ rows contain the coefficients of $f$ padded with zeroes and shifted respectively on the right by $0, 1, \ldots, m-1$ positions, while the remaining $n$ rows are made of the coefficients of $g$ padded with zeroes and shifted respectively on the right by $0,1, \ldots, n-1$ positions. 


The resultant may be directly constructed from the roots $f$ and $g$, as follows.

\begin{prop}[{\cite[Prop. 8.3]{Lang}}] \label{prop:Resultant}
Let $f,g \in \k[x]$ as above, and let $L$ be an extension of $\k$ where both $f$ and $g$ split completely, i.e.
\begin{align*}
    f &= a_n \left(x-\alpha_1\right) \cdots \left(x-\alpha_n\right) \in L[x], \\
    g &= b_m \left(x-\beta_1\right) \cdots \left(x-\beta_m\right) \in L[x].
\end{align*}
Then 
\[
    R(f,g) = a_n^m b_m^n \prod_{i=1}^n \prod_{j=1}^m \left(\alpha_i - \beta_j \right).
\]
\end{prop}

This formula leads to useful corollaries.

\begin{cor}[{\cite[Cor. 8.4]{Lang}}]
Let $f,g \in \k[x]$ as above.
Then $R(f,g)=0$ if and only if $f$ and $g$ have a common root in some field extension of $\k$.
\end{cor}

\begin{cor}[{\cite[p. 203]{Lang}}] \label{cor:Resultant}
Let $f,g \in \k[x]$ as above, then 
\[
R(f,g) = a_n^m \prod_{i=1}^n g(\alpha_i), \quad
R(f,g) = (-1)^{nm} b_m^n \prod_{j=1}^m f(\beta_j) .
\]
\end{cor}

We will apply resultants for constructing minimal polynomials of composite extensions.
In this perspective, we employ it to define another polynomial in $\k[x]$. 

\begin{notat} \label{not:Rfg}
Let $f,g \in \k[x]$ as above. For every $y \in \k$ we denote
\[
    R_{f,g}(y) = R\big(f(x),g(y-x)\big).
\]
We can view it as a polynomial $R_{f,g}(y) \in \k[y]$, which can again be seen as a polynomial $R_{f,g} \in \k[x]$ by evaluating $y$ in $x$.
Finally, we will drop the indices $f$ and $g$ when they are clear from the context.
\end{notat}

\begin{prop} \label{rescomp}
Let $f,g \in \k[x]$ be monic with $n =\deg(f)$, $m = \deg(g)$, and let $\alpha_1, \ldots, \alpha_n$ and $\beta_1, \ldots, \beta_m$ their respective (not necessarily distinct) roots in an extension $L$ of $\k$.
Then
\begin{equation*}
    R_{f,g} = \prod_{i=1}^n \prod_{j=1}^m (x-\alpha_i-\beta_j). 
\end{equation*}
\end{prop}

\begin{proof}
 Since $g= \prod_{j=1}^m (x- \beta_j) \in L[x]$, then for every $y \in \k$ we have $g(y-x)= \prod_{j=1}^m (y-x-\beta_j)$.
 From Corollary \ref{cor:Resultant} we obtain 
 \begin{equation*}
 R_{f,g}(y) = R\big(f(x),g(y-x)\big) = \prod_{i=1}^n g(y-\alpha_i) 
 = \prod_{i=1}^n \prod_{j=1}^m (y-\alpha_i- \beta_j),
 \end{equation*}
 which evaluated in $x$ as in Notation \ref{not:Rfg} gives the desired result.
\end{proof}

\begin{remark}
 It immediately follows from definitions that
 \[
    R\big(g(y-x),f(x)\big) = (-1)^{nm} R_{f,g}(y) = R_{g,f}(y).
 \]
\end{remark}

 
\subsection{Linearly disjoint extensions}

In this section, we recall the basics of field extensions that will be considered in the present paper.

\begin{prop}[{\cite[Ch. 5, Prop. 5.1]{Cohn}}]
\label{lin.indep.cond.}
Let $\k$ be a field and $\Omega$ be an algebraic extension of $\k$.
Let $A$ and $B$ be $\k$-subalgebras of $\Omega$. The following conditions are equivalent:
\begin{itemize}
    \item The $\k$-algebra homomorphism defined by
    \begin{equation*}
        A \otimes_{\k} B \to \Omega, \quad
        a \otimes b \mapsto ab,
    \end{equation*}
    is injective.
    \item Any $\k$-basis of $A$ is linearly independent over $B$.
    \item Any $\k$-basis of $B$ is linearly independent over $A$.
    \item If $\{u_i\}_{i}$ is a $\k$-basis of $A$ and $\{v_j\}_{j}$ is a $\k$-basis of $B$, then $\{u_iv_j\}_{i,j}$ are $\k$-linearly independent.
\end{itemize}
\end{prop}

In this work, we will always consider $\k = \QQ$, $A$ and $B$ will be number fields (seen as subfields of $\C$ after a fixed field embedding), and $\Omega$ will be their compositum $AB$, namely the smallest number field containing both $A$ and $B$.

\begin{defn}[Linearly disjointness]
Two number fields satisfying any (equiv. every) conditions of Proposition \ref{lin.indep.cond.} are called \emph{linearly disjoint}.
\end{defn}

The simplest way to detect linearly disjointness is by looking at the composite degree. For the reader's convenience, we recall the proof of this fact.

\begin{lemma} \label{CompositeDegree}
Two number fields $L_1$ and $L_2$ are linearly disjoint if and only if
\begin{equation*}
    [L_1L_2 : \QQ] = [L_1 : \QQ][L_2 : \QQ].
\end{equation*}
\end{lemma}
\proof
Let $\{u_i\}_{1 \leq i \leq [L_1 : \QQ]}$ be a $\QQ$-basis of $L_1$ and $\{v_j\}_{1 \leq j \leq [L_2 : \QQ]}$ be a $\QQ$-basis of $L_2$.
By definition of compositum, we have
\begin{equation*}
    L_1L_2 = \langle \{u_i v_j\}_{i,j} \rangle_\QQ .
\end{equation*}
The fields $L_1$ and $L_2$ are linearly disjoint if and only if $\{u_i v_j\}_{i,j}$ are $\QQ$-linearly independent, i.e. they generate a space of dimension $[L_1 : \QQ][L_2 : \QQ]$ over $\QQ$.
\endproof

From the above lemma, it is easy to see that when $L_1$ and $L_2$ are linearly disjoint, then $L_1 \cap L_2 = \QQ$.
If at least one of them is normal, the opposite implication also holds.

\begin{prop}[{\cite[Thm. 5.5]{Cohn}}]
\label{prop:lin.disj}
Let $L_1, L_2$ be number fields, of which at least one is a normal extension of $\QQ$. Then they are linear disjoint if and only if
\begin{equation*}
    L_1 \cap L_2 = \QQ.
\end{equation*}
\end{prop}

It is well-known that if the discriminants of two number fields $L_1,L_2$ are coprime, then they are linearly disjoint.
The opposite also holds whenever $\O_{L_1L_2} = \O_{L_1}\O_{L_2}$ \cite{CoprimeDiscriminant}.

A primitive element of the compositum of linearly disjoint fields may be easily characterized.

\begin{prop}\label{compositum}
Let $\QQ(\a)$ and $\QQ(\b)$ be linearly disjoint number fields.
Then their compositum is $\QQ(\a+\b)$.
\end{prop}
\proof It follows from \cite{Isaacs}, by noticing that the coprimality assumption may be replaced in the whole proof of the theorem by the condition of Lemma \ref{CompositeDegree}.
\endproof

\begin{cor} \label{corcomp}
Let $\QQ(\a)$ and $\QQ(\b)$ be two linearly disjoint number fields and let $f,g \in \QQ[x]$ be minimal polynomials of $\a$ and $\b$ over $\QQ$.
Then a defining polynomial for $\QQ(\a,\b)$ is $R_{f,g}$. 
\end{cor}
\begin{proof}
Let $n=[\QQ(\a):\QQ]=\deg(f)$ and $m=[\QQ(\b):\QQ]=\deg(g)$, and let $h \in \QQ[x]$ be the minimal polynomial of $\a+\b$ over $\QQ$.
Proposition \ref{compositum} ensures that $\QQ(\a,\b)=\QQ(\a+\b)$ and since the number fields are linearly disjoint, from Lemma \ref{CompositeDegree} we know that $mn=[\QQ(\a+\b):\QQ]=\deg(h)$.
From Proposition~\ref{rescomp} the polynomial $R_{f,g}$ is monic, has degree $nm$ and $\a+\b$ is one of its roots, then $h \mid R_{f,g}$.
Since they have the same degree, we necessarily have $h=R_{f,g}$.
\end{proof}

By means of Corollary \ref{corcomp}, we will always regard the compositum of two linearly disjoint number fields $\QQ[x]/(f)$ and $\QQ[x]/(g)$ as the field generated by their resultant, namely $\QQ[x]/(R_{f,g})$.

\begin{remark}
    Even when $R_{f,g}$ is generating the compositum $\QQ(\a+\b)$, it is not guaranteed to be an optimal generator. In fact, the minimal polynomials of elements $\{\a + k \b\}_{k \in \Z}$ tend to have large coefficients \cite[Remark to Algorithm 2.1.8]{Cohen}.
\end{remark}

 
\section{First-Degree Prime Ideals} \label{sec:FDPI}

We consider the following setting: let $\QQ(\a)$ and $\QQ(\b)$ be two linearly disjoint number fields and let $f \in \ZZ[x]$ (resp. $g \in \ZZ[x]$) be the minimal polynomial of $\a$ (resp. $\b$) over $\QQ$. 
We also consider the compositum $\QQ(\a,\b)$, which is equal to $\QQ(\a+\b)$ by Proposition \ref{compositum}.
Let $L$ be a field extension of the field $\k$, we will denote by $N_{L/\k}(x)$ the norm of the element $x \in L$ over the field $\k$.
Given an algebraic integer $\t$, we also recall that the norm of a non-zero ideal $\Ia \subseteq \ZZ[\t]$ is
\[
    \Norm(\Ia) = [\ZZ[\t] : \Ia].
\]

\begin{defn}[First-degree prime ideals]
Let $\t \in \C$ be an algebraic integer.
A non-zero prime ideal $\Ip$ of $\ZZ[\t]$ is called a \emph{first-degree prime ideal} if $\Norm(\Ip)$ is a prime integer.
\end{defn}
It is possible to give an explicit characterization of this particular family of ideals.

\begin{thm}\textnormal{(\cite[pp. 58--59]{BLP})}
Let $f \in \Z[x]$ be an irreducible monic polynomial and $\t \in \C$ one of its roots.
Then, for every integer prime $p$ there is a bijection between

\begin{equation*}
    \{(r,p) \ | \ r \in \FF_p \text{ such that } f(r) = 0 \in \FF_p\}
\end{equation*}
and

\begin{equation*}
    \{ \Ip \ | \ \Ip \in \textnormal{Spec}\ \Z[\t] \text{ such that } \Norm(\Ip) = p \}.
\end{equation*}
\end{thm}

The bijection considered in the previous theorem is given by the evaluation of $\t$ in a root $r$ of $f \bmod p$, namely such ideals $\Ip$ arise as kernels of the evaluations
\[
    \text{ev}_{\t \mapsto r}: \ZZ[\t] \to \FF_p, \quad \t \mapsto r.
\]

The division of ideals in $\ZZ[\t]$ using only first-degree prime ideals is completely addressed in \cite{BLP}, as it is one of the main tools on which the GNFS relies.
For a quick recap on these properties, see \cite[Section 2]{SanTau}.

Here we are interested in studying the relation among first-degree prime ideals of the orders $\ZZ[\a]$, $\ZZ[\b]$ and those of $\ZZ[\a+\b]$. 
The following result shows that it is always possible to efficiently construct first-degree prime ideals of $\ZZ[\a+\b]$ starting from those of $\ZZ[\a]$ and $\ZZ[\b]$.

 \begin{prop} \label{idealsdownup}
Let $(r,p)$ be a first-degree prime ideal of $\ZZ[\a]$ and $(s,p)$ be a first-degree prime ideal of $\ZZ[\b]$, then $(r+s,p)$ is a first-degree prime ideal of $\ZZ[\a+\b]$.
 \end{prop}
 \begin{proof}
 From Corollary \ref{corcomp}, we know that the minimal polynomial of $\a+\b$ is $R_{f,g}$.
 Since $(r,p)$ is a first-degree prime ideal of $\ZZ[\a]$, then $r$ is a root of $f \bmod{p}$. Analogously, $s$ is a root of $g \bmod{p}$. The definition of $R_{f,g}$ as seen in Proposition \ref{rescomp} leads to the desired result.
 \end{proof}
 
 \begin{remark}
  The previous result applied to biquadratic extensions is precisely \cite[Theorem 2]{SanTau}.
 \end{remark}
 
Proposition \ref{idealsdownup} motivates the following definition.

 \begin{defn}[Combination]
  We say that the first-degree prime ideal $(r+s,p) \subseteq \ZZ[\a+\b]$ is the \emph{combination} of $(r,p) \subseteq \ZZ[\a]$ and $(s,p) \subseteq \ZZ[\b]$. 
 \end{defn}

 The following proposition shows that almost every first-degree prime ideal of $\ZZ[\a+\b]$ arise from a combination of first-degree prime ideals of $\ZZ[\a]$ and $\ZZ[\b]$.

 \begin{prop} \label{updown1}
 Let $(t,p)$ be a first-degree prime ideal of $\ZZ[\a+\b]$, where $t$ is a simple root of $R_{f,g} \bmod p$.
 Then $(t,p)$ is a combination of first-degree prime ideals of $\ZZ[\a]$ and $\ZZ[\b]$.
 \end{prop}
 \begin{proof}
 Let $\FF_q$ be an extension of $\FF_p$ where both $f \bmod p$ and $g \bmod p$ split.
 By Proposition~\ref{rescomp} the roots of $R=R_{f,g} \bmod p$ are sums of roots in $\FF_q$ of $f \bmod p$ and $g \bmod p$, i.e. there are $\c_1, \c_2 \in \FF_q$ such that $t = \c_1 + \c_2$ and
 \[
    f(\c_1) = 0 = g(\c_2).
 \]
 It is well-known \cite[Theorem 2.14]{Lidl} that the conjugates of $\c$ over $\FF_p$, which belong to the set $\{\c^{p^n} : n \in \NN\}$, are simple roots of the same irreducible polynomial, hence
 \[
    f(\c_1^p) = 0 = g(\c_2^p).
 \]
 Therefore, also $\c_1^p + \c_2^p$ is a root of $R$. However, we have
 \[
    \c_1^p + \c_2^p = (\c_1 + \c_2)^p = t^p = t.
 \]
 Thus, either $t$ is a multiple root of $R$ or all the conjugates of $\c_1$ are equal, and so are those of $\c_2$.
 Since, by hypothesis, we are in the latter case, both $\c_1$ and $\c_2$ belong to $\FF_p$, so $(t,p)$ is the combination of $(\c_1,p)$ and $(\c_2,p)$.
 \end{proof}
 
 \begin{remark} \label{rmk:AlmostAlways}
 The resultant polynomial $R_{f,g}$ is irreducible over $\QQ$ by Corollary \ref{corcomp}, then it has no repeated roots.
 Hence, its discriminant $R( R_{f,g}, R_{f,g}' )$ is a non-zero integer, which is therefore divisible only by a finite set $\mathcal{P}$ of integer primes.
 In particular, for every prime $p \not\in \mathcal{P}$, the projected resultant $R_{f,g} \bmod p$ will have only simple roots, so every prime ideal of $\Z[\a+\b]$ of norm $p$ arises a combination of first-degree prime ideals in $\Z[\a]$ and $\Z[\b]$ by Proposition \ref{updown1}.
 For a more precise description of this set $\mathcal{P}$ we refer to \cite[Lemma 2.1.13]{Cohen}
 \end{remark}

 \begin{remark}
 We notice that Proposition \ref{updown1} generalizes \cite[Theorem 3]{SanTau}.
 In fact, let us consider $f(x)=x^2-a$ and $g(x)=x^2-b$, and let $p$ be a prime and $\c_1,\c_2 \in \FF_{p^2}$ such that $f(\c_1) = 0 = g(\c_2)$.
 It is clear that also $-\c_1$ (resp. $-\c_2$) is a root of $f$ (resp. $g$), therefore the roots of $R_{f,g}$ in $\FF_{p^2}$ are $\pm \c_1 \pm \c_2$.
 An easy check shows that $R_{f,g}$ has a multiple root if and only if
 \begin{itemize}
     \item $p = 2$, or
     \item $\c_1 = 0$ or $\c_2 = 0$, or
     \item $t = \c_1 + \c_2 = 0$.
 \end{itemize}
 In the first two cases, the first-degree prime ideal $(t,p) \subseteq \ZZ[\a+\b]$ arises anyway as a combination, while when $t = 0$ this does not necessarily hold \cite[Example 3]{SanTau}.
 \end{remark}
 
 We now prove that when $\QQ(\a)$ and $\QQ(\b)$ are both normal and of coprime degrees, we are guaranteed that every first-degree prime ideal of $\ZZ[\a+\b]$ arises as a combination, without exceptions.
 First, we prove a technical result linking a global property of polynomials with the degrees of their local factors.
 It is stated independently on the following results, as it has its own theoretical interest.
 
 \begin{prop}
 \label{prop:dividingDegree}
 Let $f \in \ZZ[x]$ and let $L$ be its splitting field over $\QQ$.
 Let $p$ be an integer prime and $h \in \FF_p[x]$ be an irreducible factor of $f \bmod p$.
 Then
 \[
    \deg h  \ | \ [ L : \QQ ].
 \]
 \end{prop}
 \proof Let $\O_L$ be the ring of integers of $L$ over $\QQ$ and let $\frakp \subseteq \O_L$ be a prime lying over $p$.
 It is well-known that, since $L/\QQ$ is Galois, the ramification index $\frake$ and the inertia degree $\frakf$ are independent of $\frakp$. Thus, if $\frakg$ is the number of primes lying over $p$, we have
 \[
    [ L : \QQ ] = \frake\frakf\frakg,
 \]
 and in particular $\frakf | [ L : \QQ ]$.
 Since $f$ splits in $\O_L$, it also splits in $\O_L/\frakp$, hence this extension of $\FF_p$ contains the splitting field of $f$ over $\FF_p$.
 Since $h$ is irreducible, $\O_L/\frakp$ also contains the field $\FF_p[x]/(h)$, which has degree $\deg h$ over $\FF_p$.
 Therefore, we have
 \[
    \deg h \ | \ [\O_L/\frakp : \FF_p] = \frakf,
 \]
 which concludes the proof.
 \endproof
 
 We can now prove the combination result.
 
 \begin{prop} \label{updown2}
 Let $f,g \in \ZZ[x]$ be monic, irreducible polynomials of coprime degrees such that $\QQ(\a) = \QQ[x]/(f)$ and $\QQ(\b) = \QQ[x]/(g)$ are normal extensions of $\QQ$.
 If $(t,p)$ is a first-degree prime ideal of $\ZZ[\a+\b]$, then it is a combination of first-degree prime ideals of $\ZZ[\a]$ and $\ZZ[\b]$.
 \label{prop:soprasotto}
 \end{prop}
 \proof Since the degrees are coprime we have $\QQ(\a) \cap \QQ(\b) = \QQ$, and since $\QQ(\a)$ and $\QQ(\b)$ are normal by Proposition \ref{prop:lin.disj} we know that they are linearly disjoint. 
 Thus, by Corollary \ref{corcomp} their compositum $\QQ(\a,\b)$ is generated by $R = R_{f,g}$, and by hypothesis we have
 \[
    R(t) \equiv 0 \bmod p.
 \]
 Let $\overline{f}, \overline{g} \in \FF_p[x]$ be the projections of $f$ and $g$ modulo $p$, and let $\FF_q$ be their common splitting field.
 By Proposition \ref{rescomp} there are $\nu, \mu \in \FF_q$ such that 
 \[
    \overline{f}(\nu) = 0, \qquad \overline{g}(\mu) = 0, \qquad t = \nu + \mu.
 \]
 Let $h_f$ and $h_g$ be minimal polynomials of $\nu$ and $\mu$ over $\FF_p$, respectively.
 Since $L_1$ and $L_2$ are normal over $\QQ$, Proposition \ref{prop:dividingDegree} implies that
 \[
    \deg h_f \ | \ \deg f, \qquad \deg h_g \ | \ \deg g.
 \]
 Since $\deg f$ and $\deg g$ are coprime, also $\gcd(\deg h_f,\deg h_g)=1$.
 However, since $\nu + \mu = t \in \FF_p$ we have $\FF_p(\nu) = \FF_p(\mu)$. This may only happen if
 \[
    \FF_p(\nu) = \FF_p(\mu) = \FF_p,
 \]
 which means that $\nu, \mu \in \FF_p$.
 Hence, we conclude that $(t,p)$ is the combination of $(\nu,p)$ and $(\mu,p)$.
 \endproof
 
The following examples show that both the normality and the coprimality of degrees are necessary conditions for Proposition \ref{updown2}.
 
 \begin{exmp} \label{ex:only1normal}
 Let us consider the following irreducible polynomials
 \[
     f(x) = x^2 - 3, \qquad
     g(x) = x^3 - 2,
 \]
 and let $\QQ(\a) = \QQ[x]/(f)$ and $\QQ(\b) = \QQ[x]/(g)$ be the number fields they generate.
 We notice that the degrees are coprime and $\QQ(\a)$ is Galois, while $\QQ(\b)$ is not normal.
 A defining polynomial of the compositum $\QQ(\a+\b)$ is the resultant 
 \[
    R_{f,g} = x^6-9x^4-4x^3+27x^2-36x-23.
 \]
 The first-degree prime ideals of $\ZZ[\a]$ and $\ZZ[\b]$ with norm $17$ correspond to the roots modulo $17$ of $f$ and $g$.
 One can directly verify that there are none of them in $\ZZ[\a]$, while $(8,17)$ is a first-degree prime ideal in $\ZZ[\b]$.
 However, $(13,17) \subseteq \ZZ[\a+\b]$ is a first-degree prime ideal of norm $17$, which cannot be a combination of first-degree prime ideals in the underlying extensions.
 This shows that the hypothesis of normality on both extensions is necessary for Proposition \ref{prop:soprasotto}. 
 \end{exmp}
 
 \begin{exmp}
 Let $f$ be as in Example \ref{ex:only1normal} and consider $g=x^4+1$.
 These polynomials are both irreducible over $\QQ$ and generate normal extensions $\QQ(a)$ and $\QQ(\b)$.
 The compositum $\QQ(\a+\b)$ is defined by the polynomial
 \[
    R_{f,g} = x^8-12x^6+56x^4-72x^2+100.
 \]
 Neither $\ZZ[\a]$ nor $\ZZ[\b]$ have first-degree prime ideals with norm $5$, although there is a first-degree prime ideal in $\ZZ[\alpha+\beta]$ of norm $5$, that is $(0,5)$, which again cannot arise from any combination of first-degree prime ideals in the underlying extensions.
 Therefore we also need coprime degrees in Proposition \ref{prop:soprasotto}. 
 \end{exmp}
 
 
 \section{Divisibility of prescribed principal ideals} \label{sec:Divisibility}
 
 Given an algebraic integer $\t$, it is known that the prime factors of principal ideals of the form $(e+d\t)\ZZ[\t]$ with $\gcd(e,d)=1$ are all first-degree primes $(t,p) \subseteq \ZZ[\t]$ such that $e+dt \equiv 0 \bmod p$ \cite[Corollary 5.5]{BLP}.
 In this section, we detail how this divisibility can be read from the underlying fields, and vice-versa.
 The results presented in \cite[Section 4]{SanTau} may therefore be seen as particular instances of those discussed in the present section.
 To pursue this direction, we first need to characterize the intersection of this principal ideal with the underlying rings $\ZZ[\a]$ and $\ZZ[\b]$.
 
 \begin{thm} \label{thm:intersId}
 Let $\a,\b \in \C$ be algebraic integers defining linearly disjoint number fields $\QQ(\a)$ and $\QQ(\b)$, and let $g=\sum_{i=0}^m b_ix^i \in \ZZ[x]$ be the minimal polynomial of $\b$ over $\QQ$.
 Let $e,d \in \ZZ$ be coprime integers and let $I$ be the principal ideal generated by $\xi=e+d(\a+\b)$ in $\ZZ[\a+\b]$.
 Then
 \[ 
    I \cap \ZZ[\a] = (\chi)\ZZ[\a]
 \]
 is still principal, generated by $\chi = N_{\QQ(\a+\b)/\QQ(\a)} \left( \xi \right)$, namely
 \[
    \chi = \sum_{i=0}^m (-d)^i \Omega^{m-i} b_{m-i}, \qquad \text{where} \qquad \Omega=e+d\alpha \in \ZZ[\a].
 \]
 \end{thm}
 
 \begin{proof}
 \begin{itemize}
     \item[($\subseteq$)] Since $\QQ(\a)$ and $\QQ(\b)$ are linearly disjoint, then $\{1,\b, \ldots, \b^{m-1} \}$ is a basis for $\QQ(\a+\b)$ over $\QQ(\a)$ by \cref{lin.indep.cond.}.
     Thus, every $z \in I$ may be written as
     \[
        z = (\Omega + d\b) \left(\lambda_0+\lambda_1\b+\ldots+\lambda_{m-1}\b^{m-1}\right),
     \]
     where $\lambda_0, \ldots, \lambda_{m-1} \in \ZZ[\a]$. 
     If we also require $z \in \ZZ[\a]$, an explicit computation gives 
     \begin{equation}
     \begin{cases}
     \lambda_1 \Omega +d \lambda_0 - \lambda_{m-1}db_1 =0, \\
     \lambda_2 \Omega +d \lambda_1 - \lambda_{m-1}db_2 =0, \\
     \quad \vdots\\
     \lambda_{m-2} \Omega +d \lambda_{m-3} - \lambda_{m-1}db_{m-2} =0, \\
     \lambda_{m-1} \Omega +d \lambda_{m-2} - \lambda_{m-1}db_{m-1} =0.
     \end{cases}
     \label{eq:system1}
     \end{equation}
     We first prove that for every $0 \leq i \leq m-1$ we have $d^i \mid \lambda_i$.
     To do so, we prove by induction on $0 \leq j \leq i$ that $d^j \mid \lambda_i$.
     The base step $j = 0$ is trivial.
     Let us assume that $d^j \mid \lambda_i$ for all $j \leq i \leq m-2$. 
     For every $1 \leq k \leq i-j$, the $(j+k)$-th equation of system \eqref{eq:system1} gives
     \[
        e\lambda_{j+k} = d\left(\lambda_{m-1}b_{j+k}-\lambda_{j+k-1} - \a\lambda_{j+k} \right).
     \]
     Since $(e,d) = 1$ and by induction $d^j \mid \lambda_{m-1}b_{j+k}-\lambda_{j+k-1} - \alpha\lambda_{j+k}$, then $d^{j+1} \mid \lambda_{j+k}$ for every $1 \leq k \leq i-j$, i.e. $d^{j+1} \mid \lambda_i$ whenever $j+1 \leq i$.
    We now prove by induction on $2 \leq k \leq m$ that
     \begin{equation}
         \lambda_{m-k} = \frac{\lambda_{m-1}}{d^{k-1}} \left(\sum_{j=0}^{k-1} d^j (-\Omega)^{k-1-j} b_{m-j} \right),
         \label{eq:formulalambda}
     \end{equation}
     which is well-defined since $\frac{\lambda_{m-1}}{d^{k-1}} \in \ZZ$, as noted before.
     The base step $k = 2$ is given by the last equation of \eqref{eq:system1}, indeed
     \begin{equation*}
        \lambda_{m-2} = \frac{\lambda_{m-1}}{d} \left(db_{m-1}-\Omega\right).
     \end{equation*}
     We now suppose that \eqref{eq:formulalambda} holds for $k \leq m-1$ and check that this implies it for $k+1$. From the $(m-k)$-th equation of the system \eqref{eq:system1} we have
     \[
        \lambda_{m-k} \Omega +d \lambda_{m-k-1} -d\lambda_{m-1}b_{m-k}=0,
     \]
     which by inductive hypothesis becomes
     \begin{align*}
     \lambda_{m-k-1} &= \frac{1}{d} \left( d \lambda_{m-1}b_{m-k} +  \frac{\lambda_{m-1}}{d^{k-1}} \left(\sum_{j=0}^{k-1} b_{m-j} d^j (-\Omega)^{k-j} \right) \right) \\
     &=  \frac{\lambda_{m-1}}{d^{k}} \left(d^kb_{m-k} + \sum_{j=0}^{k-1} b_{m-j} d^j (-\Omega)^{k-j} \right) \\
     &= \frac{\lambda_{m-1}}{d^{k}} \left(\sum_{j=0}^{k} b_{m-j} d^j (-\Omega)^{k-j} \right),
     \end{align*}
     This proves that \eqref{eq:formulalambda} holds, and in particular
     \begin{equation}
           \lambda_{0} = \frac{\lambda_{m-1}}{d^{m-1}} \left(\sum_{j=0}^{m-1} b_{m-j} d^j (-\Omega)^{m-1-j} \right).
           \label{eq:lambda0}
     \end{equation}
     
     When system \eqref{eq:system1} holds, we have $z=\lambda_0\Omega - \lambda_{m-1}db_0$, which by means of \eqref{eq:lambda0} can be written as
     \begin{align*}
        \lambda_0\Omega - \lambda_{m-1}db_0 &= \frac{\lambda_{m-1}}{d^{m-1}} \left(\sum_{j=0}^{m-1} b_{m-j} d^j (-\Omega)^{m-1-j} \right) \Omega - \lambda_{m-1}db_0 \\
        &=\frac{\lambda_{m-1}}{d^{m-1}} \left(\sum_{j=0}^{m-1} b_{m-j} d^j (-1)^{m+1-j} \Omega^{m-j} - d^m b_0 \right) \\
        &=(-1)^{m+1}\frac{\lambda_{m-1}}{d^{m-1}} \left(\sum_{j=0}^{m-1} b_{m-j} (-d)^j \Omega^{m-j} + (-d)^m b_0 \right)\\
        &= (-1)^{m+1}\frac{\lambda_{m-1}}{d^{m-1}} \chi.
        \end{align*}
        Since $\frac{\lambda_{m-1}}{d^{m-1}} \in \ZZ[\a]$, then $z \in (\chi)\ZZ[\a]$.
        
        \item[($\supseteq$)] By definition $\chi \in \ZZ[\a]$, and by a straightforward computation we get
        \begin{equation}
            \chi = \prod_{i=1}^m \left( \Omega + d\b_i \right) = N_{\QQ(\a+\b)/\QQ(\a)} \left( \xi \right),
            \label{eq:norm}
        \end{equation}
        where the $\b_i$'s are the roots of $g(x)$ in its splitting field.
        Since $\xi \in \ZZ[\a+\b] \subseteq \mathcal{O}_{\QQ(\a+\b)}$, then it satisfies a polynomial with coefficients in $\ZZ[\a]$, namely there are $h_i \in \ZZ[\a]$ such that 
        \[
            h(\xi)=h_{t}\xi^{t} + h_{t-1}\xi^{t-1} + \ldots + h_0 = 0.
        \]
        Then 
         \begin{align*}
            \chi=N_{\QQ(\a+\b)/\QQ(\a)}(\xi) = (-1)^t h_0 = (-1)^{t+1} \xi \left(h_t \xi^{t-1} + h_{t-1}\xi^{t-2}+ \ldots + h_1 \right),
         \end{align*}
         so it belongs to $(\xi)\ZZ[\a+\b]$.
 \end{itemize}
 \end{proof}
 
 \begin{remark}
 It is easy to verify that the biquadratic case discussed in \cite[Proposition 4]{SanTau} is simply an instance of Theorem \ref{thm:intersId}, when $\b^2 \in \ZZ$ and $g = x^2-\b^2$.
 \end{remark}
 
 We now fix some notation: let $\a,\b \in \C$ be algebraic integers such that $\QQ(\a)$ and $\QQ(\b)$ are linearly disjoint number fields, let $e,d \in \ZZ$ be coprime integers and let us consider the principal ideal $I = \big(e + d (\a+\b)\big) \subseteq \ZZ[\a+\b]$.
 Let also $f=\sum_{i=0}^n a_ix^i \in \ZZ[x]$ be the minimal polynomial of $\a$ and $g=\sum_{i=0}^m b_ix^i \in \ZZ[x]$ be the minimal polynomial of $\b$.
 By Theorem \ref{thm:intersId} we know that
 \[
    I_\a = I \cap \ZZ[\a] = (\chi_\a)\ZZ[\a], \quad \text{where} \quad \chi_\a = \sum_{i=0}^m (-d)^i (e+d\a)^{m-i} b_{m-i},
 \]
 and
 \[
    I_\b = I \cap \ZZ[\b] = (\chi_\b)\ZZ[\b], \quad \text{where} \quad \chi_\b = \sum_{i=0}^n (-d)^i (e+d\b)^{n-i} a_{n-i}.
 \] 
 Finally, whenever $p$ is a prime not dividing $d$, we may define the affine map
 \[
    \phi : \FF_p \to \FF_p, \qquad x \mapsto -x - d^{-1}e.
 \]
 
 \begin{thm} 
 \label{thm:divbeltoab}
 In the above notation, let $(r,p)$ be a first-degree prime of $\ZZ[\a]$ dividing $I_\a$ and $(s,p)$ be a first-degree prime of $\ZZ[\a]$ dividing $I_\b$.
 Then $(r+s,p)$ is a first-degree prime of $\ZZ[\a+\b]$ dividing $I$, unless $\phi(r)$ is a root of $g \bmod p$ different from $s$ and, at the same time, $\phi(s)$ is a root of $f \bmod p$ different from $r$.
 \end{thm}
 \proof Since $(r,p) | I_\a$, we have
 \[
    \sum_{i = 0}^m (-d)^i(e+dr)^{m-i} b_{m-i} \equiv 0 \bmod p.
 \]
 If $d \equiv 0 \bmod p$, the above equation leads to $e^m \equiv 0 \bmod p$, contradicting the coprimality of $e$ and $d$. Hence, we may assume $d \not\equiv 0 \bmod p$, and write
 \[
    \sum_{i = 0}^m (-d)^i(e+dr)^{m-i} b_{m-i} = (-d)^m g\left( \frac{e+dr}{-d} \right) = (-d)^m g\big( \phi(r) \big).
 \]
 Since $p \nmid d$, this implies that $\phi(r)$ is a root of $g \bmod p$.
 The same argument also shows that $\phi(s)$ needs to be a root of $f \bmod p$.
 By hypothesis we may assume that either $\phi(r) = s$ or $\phi(s) = r$, which both imply 
 \[
    r+s + d^{-1}e \equiv 0 \bmod p.
 \]
 Since $I$ is generated by $e+d(\a+\b)$, the above congruence shows that the combination $(r+s,p)$, which is a first-degree prime ideal of $\ZZ[\a+\b]$ by Proposition \ref{idealsdownup}, divides $I$.
 \endproof
 
 The condition $\phi(r) \neq s$ being a root of $g \bmod p$ and $\phi(s) \neq r$ being a root of $f \bmod p$ of Theorem \ref{thm:divbeltoab} will be referred to as the \emph{exceptional case}.
 It appears to be extremely rare, especially when the considered extensions are small (e.g. see Proposition \ref{prop:case6}).
 However, it may occasionally occur and it might not be evident \emph{a priori}, as shown by the following example.
 
 \begin{exmp}
 Let us consider the polynomials
 \[
    f = x^3 + x^2 + x + 19, \qquad g = x^4 - 6 x^2 - 7 x + 5,
 \]
 generating the number fields $\QQ(\a)$ and $\QQ(\b)$, whose composite $\QQ(\t)$ is generated by
 \begin{align*}
     h = \ & x^{12} + 4 x^{11} - 8 x^{10} + 11 x^9 + 193 x^8 + 824 x^7 + 5663 x^6 + 8910 x^5 + 32405 x^4 + 120009 x^3 \\
     & + 185557 x^2 + 255445 x + 24299.
 \end{align*}
 Let us consider the principal ideal
 \[
    I = (1 + \t)\ZZ[\t],
 \]
 whose intersections with $\ZZ[\a]$ and $\ZZ[\b]$ are generated by 
 \[
    \chi_\a = -4 \a^2 - 23 \a - 50, \qquad \chi_\b = \b^3 + 2 \b^2 + 2 \b - 18.
 \]
 We observe that $(1, 11),(2, 11),(7, 11) \subseteq \ZZ[\a]$ are first-degree prime ideals of $\ZZ[\a]$, while the norm-$11$ first-degree primes of $\ZZ[\b]$ are
 $(3,11), (9,11) \subseteq \ZZ[\b]$.
 However, we have
 \[
    \phi(1) \equiv 9 \bmod 11, \qquad \phi(3) \equiv 7 \bmod 11.
 \]
 Hence, we are in the exceptional case of Theorem \ref{thm:divbeltoab}: the first-degree prime ideal $(4,11) \subseteq \ZZ[\t]$ given by the combination of $(1, 11) \in \ZZ[\a]$ and $(3,11) \in \ZZ[\b]$ does not divide $I$, as
 \[
    1+(1+3) \equiv 5 \not\equiv 0 \bmod 11.
 \]
 \end{exmp}
 
 \begin{remark}
 We highlight that Theorem \ref{thm:divbeltoab} generalizes \cite[Theorem 4]{SanTau}.
 In fact, when $f(x)=x^2-a$ and $g(x)=x^2-b$, the exceptional case occurs only if
 \[\begin{cases}
    e+dr \equiv ds \bmod p,\\
    e+ds \equiv dr \bmod p.
 \end{cases}\]
 If $p = 2$, these equations are both equivalent to $e+d(r+s) \equiv 0 \bmod p$.
 If $p \neq 2$, they imply $e \equiv 0 \bmod p$, but in this case $e+d(r+s) \equiv 0 \bmod p$ they still hold when $r+s \equiv 0 \bmod p$.
 Therefore, the only exceptions may arise when $p \neq 2$, $e \equiv 0 \bmod p$ and $r+s \not\equiv 0 \bmod p$, as prescribed by \cite[Theorem 4]{SanTau}.
 \end{remark}
 
 On the other hand, we show that if a combination divides $I$, then its constituents always divide the correspondent restrictions $I_\a$ and $I_\b$.

 \begin{thm} \label{thm:divabtobel}
 In the above notation, let $(t,p) \subseteq \ZZ[\a+\b]$ be a first-degree prime ideal dividing $I$.
 If there exist first-degree primes $(r,p) \subseteq \ZZ[\a]$ and $(s,p) \subseteq \ZZ[\b]$ such that $r+s \equiv t \bmod p$, then $(r,p)|I_\a$ and $(s,p)|I_\b$.
 \end{thm}
 \proof If $(r+s,p)$ divides the ideal generated by $e+d(\a+\b)$, then we have
 \[
    e+d(r+s) \equiv 0 \bmod p.
 \]
 Since $(d,e) = 1$, then $p \nmid d$, so we can write $r \equiv -d^{-1}e-s \bmod p$.
 Thus, we have
 \[
   \sum_{i = 0}^m (-d)^i(e+dr)^{m-i}b_{m-i} \equiv b^m g(-d^{-1}e-r) \equiv b^m g(s) \equiv 0 \bmod p,
 \]
 which proves that $(s,p) | (\chi_\b)\ZZ[\b]$.
 The proof of $(r,p) | (\chi_\a)\ZZ[\a]$ is completely analogous.
 \endproof
 
 \begin{remark}
 The result \cite[Theorem 5]{SanTau} follows by Theorem \ref{thm:divabtobel}, when the considered number fields are quadratic.
 \end{remark}
 
 We also point out that the norms of the considered ideals are equal, hence even the exponents of the first-degree divisors of the given principal ideal may be read from the underlying extensions.
 
 \begin{lemma}
 Let $\xi$, $\chi_\a$ and $\chi_\b$ defined as above, then their norms over $\QQ$ are the same, namely 
 \[
    N_{\QQ(\a+\b)/\QQ}(\xi) = N_{\QQ(\a)/\QQ}(\chi_\a) = N_{\QQ(\b)/\QQ}(\chi_\b).
 \]
 \end{lemma}
 \begin{proof}
 It follows directly from \eqref{eq:norm} and the composition of norms (see \cite[Theorem VI.5.1]{Lang}).
 \end{proof}
 
 Finally, we conclude this section by observing that for small extensions we can get rid of the exceptional cases with a few assumptions.
 As an instance, the following proposition prescribes how to produce degree-$6$ number fields where the correspondence between the first-degree prime ideals is perfect, i.e. when exceptional cases never occur.
 
\begin{prop} \label{prop:case6}
 Let $m$ be an odd integer, $\QQ(\t)$ be a Galois field of degree $2m$ and let $\QQ(\a)$ and $\QQ(\b)$ be its degree-$2$ and degree-$m$ subfields, respectively.
 Let $d,e \in \ZZ$ be coprime and $I = (e+d\t) \ZZ[\t]$.
 Then either $I \cap \ZZ[\a] = (0)$ or the first-degree prime ideals of $\ZZ[\t]$ dividing $I$ are precisely the combinations of first-degree prime ideals of $\ZZ[\a]$ and $\ZZ[\b]$ dividing $I \cap \ZZ[\a]$ and $I \cap \ZZ[\b]$, respectively.
 \end{prop}
 \begin{proof}
 We first notice that $\QQ(\a)$ and $\QQ(\b)$ are normal extensions of coprime degrees, hence by Proposition \ref{prop:lin.disj} they are linearly disjoint.
 
 On one side, by Proposition \ref{updown2} every first-degree prime ideal of $\ZZ[\t]$ arises from a combination of $(r,p) \subseteq \ZZ[\a]$ and $(s,p) \subseteq \ZZ[\b]$, and by Theorem \ref{thm:divabtobel} we know that $(r,p)|I_\a$ and $(s,p)|I_\b$.
 
 On the other side, assume that there are first-degree prime ideals $(r,p)|I_\a$ and $(s,p)|I_\b$.
 In this case $p \nmid d$, otherwise 
 \[
    0 \equiv \text{ev}_{\a \mapsto r}(\chi_\a) \equiv e^m,
 \]
 which would contradict the coprimality of $e$ and $d$.
 Since $[\QQ(\a):\QQ] = 2$, then $\chi_\a$ is a linear polynomial in $\a$.
 Thus, either $\chi_\a = 0$, or there is at most one solution $w \in \FF_p$ to
 \[
    \text{ev}_{\a \mapsto w} (\chi_\a) = (-d)^m g\big( \phi(w) \big).
 \]
 In the latter case, since $(r,p)|I_\a$ we conclude that $w = r$ is the unique zero of $\text{ev}_{\a \mapsto w} (\chi_\a)$ over $\FF_p$.
 Since $\phi$ is linear and $p \nmid d$, this implies that $s = \phi(r)$ is the unique root of $g \bmod p$, so Theorem \ref{thm:divbeltoab} applies, proving that $(r+s,p)|I$.
 \end{proof}
 
 Proposition \ref{prop:case6} notably applies for $k=3$ on sextic extensions, which are widely studied for the GNFS optimization \cite{RSA-704, CADO-NFS, RSA-768}.
 We observe that the normality condition is only necessary for ensuring that every first-degree prime of $\ZZ[\t]$ is obtained via ideal combination, but it may be dropped whenever finding them all is not a requirement.
 This is usually the case in algorithmic practice, where we are only interested in efficiently finding plenty of them. 
 Furthermore, we will computationally observe in Section \ref{sec:smooth} that the quantity of first-degree prime ideals one may miss by dropping the normality assumption is negligible, especially for primes of large norms.
 
 
 \section{Computational improvement} \label{sec:computational}
 
 In the previous sections we proved that, apart from rare exceptions, we may compute first-degree primes in composite extensions by addressing the same problem inside underlying subfields and composing the resulting solutions.
 This approach is particularly efficient for computing large sets of first-degree prime ideals in composite extensions of high degree, although consistent time improvements may also be appreciated in the well-studied degree-$6$ extensions.
 
 In the present section, we discuss the time reduction obtained from such an approach, and we computationally evaluate the results with Magma \cite{Magma}\footnote{Our testing has been performed on a personal computer running Magma V2.25-3, CPU: Intel(R) Core(TM) i7-8565U @ 1.80GHz. The MAGMA implementation may be found on GitHub at \url{https://github.com/DTaufer/First-degree-prime-ideals/blob/main/MagmaCode.m}}.

 \subsection{Asymptotic complexity}  \label{sec:Complexity}
 
 We consider a number field $\QQ(\t) = \QQ[x]/(h)$ obtained from the compositum of linearly disjoint number fields $\QQ(\a_i) = \QQ[x]/(f_i)$, and we compare the following approaches for finding first-degree prime ideals of $\ZZ[\t]$ of norm $p$.
 \FloatBarrier
 \begin{table}[!htb]
 \begin{center}
\begin{tabular}{||c || c||} 
 \hline
 Standard approach & Composite approach \\ [0.5ex] 
 \hline\hline
 Compute the roots $\mathcal{R}$ of $f \bmod p$ & Compute the roots $\mathcal{R}_i$ of $f_i \bmod p$ \\ 
 \hline
 Return $\{(r_j,p)\}_{r_j \in \mathcal{R}}$ & Return $\{(\sum_{i} r_i,p)\}_{r_i \in \mathcal{R}_i}$ \\
 \hline
\end{tabular}
\end{center}
\caption{The standard and the composite approaches for finding first-degree prime ideals.}
\label{tab:comparison1}
\end{table}
\FloatBarrier
 The complexity of both algorithms depends on the complexity of computing the roots of a given degree-$n$ polynomial over $\FF_p$, which can be achieved via the renowned Berlekamp algorithm \cite{Berlekamp}, or with more sophisticated approaches \cite{Subquadratic,RootFinding}, whose asymptotic complexity depends on the relation between $n$ and $p$.
 From a GNFS perspective, one is mostly interested in the asymptotic behavior of $p$, and the asymptotic complexity for the best-known algorithms when $p \to \infty$ is 
 \[
    O\left(n^{1+o(1)}\log p\right).
 \]
 By the Prime Number Theorem, a random positive integer $\leq M$ is prime with probability $1/\log M$, and when it is prime it requires
 $O(n^{1+o(1)}\log M)$ field operations to compute the first-degree primes of that norm.
 Thus, the computational cost of computing the first-degree prime ideals of norms $\leq M$ is expected to grow linearly with $M$.
 
 In our setting, since the underlying extensions are linearly disjoint, if $n_i = \deg(f_i)$ then $h$ may be obtained as an iterated resultant and it has degree $\deg(h) = \prod_i n_i$.
 Hence, the standard approach for finding first-degree primes in $\ZZ[\t]$ of norms $\leq M$ should require $O\big(\deg(h)^{1+o(1)} M\big)$ field operations.
 
 On the other side, solving the same problem in the smaller subfields requires repeated roots finding of degree-$n_i$ polynomials over the same base-field $\FF_{p}$, each of which can be accomplished in $O\big(n_i^{1+o(1)} p\big)$ fields operations.
 Afterward, the solutions need to be composed, which requires at most $\prod_i n_i$ additions over $\FF_p$, which does not depend on $p$ so we can neglect it.
 
 The above discussion implies that, for large values of $p$, the two approaches have the same asymptotic linear complexity.
 However, it also shows that by employing the composite approach we should expect an asymptotically linear reduction in time of about $\frac{\prod_i n_i}{\sum_i n_i}$.
 In the following sections we will computationally verify these estimates observing that, although linear, this improvement may actually be conspicuous even in small cases.

\subsection{Degree-6 extensions} \label{sec:case6}

 Here we consider degree-$6$ extensions, the degree that is often employed for the polynomial-selection phase of the GNFS \cite{RSA-704,RSA-768}.
 In the sieving phase of such an algorithm, a large set of first-degree prime ideals has to be computed to construct the \emph{algebraic factor base}.

 Every degree-$2$ polynomial is normal, and constructing degree $3$ normal polynomials is computationally effortless, hence we have decided to deal with degree-$6$ Galois extensions.
 This way, by Proposition \ref{updown2} we are guaranteed that both approaches produce the same outcome.
 
 We randomly selected ten instances of such extensions and computed the average time needed for the two aforementioned approaches to produce the first-degree prime ideals of norm $p \leq M$ for $M \leq 10^9$.
 The results are collected in Figure \ref{fig:results6}.
 
\FloatBarrier
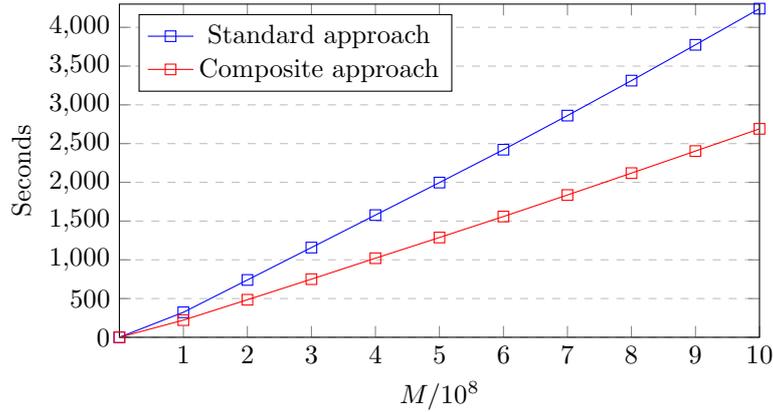
\begin{figure}[!hbt] \centering
 \begin{tikzpicture}
\begin{axis}[ width=10cm,height=6cm,
    xlabel={$M/10^8$},
    ylabel={Seconds },
    xmin=0, xmax=10, 
    ymin=0, ymax=4300, 
    xtick={1,2,3,4,5,6,7,8,9,10}, 
    ytick={0,500,1000,1500,2000,2500,3000,3500,4000}, 
    legend pos=north west,
    ymajorgrids=true,
    grid style=dashed,
]

\addplot[
    color=blue,
    mark=square,
    ]
    coordinates { 
    (0,0)(1,323.0)(2,740.2)(3,1158.1)(4,1577.4)(5,1996.9)(6,2421.4)(7,2861.4)(8,3312.6)(9,3774.7)(10,4241.3)
    };

\addplot[
    color=red,
    mark=square,
    ]
    coordinates { 
    (0,0)(1,221.7)(2,484.8)(3,750.3)(4,1020.7)(5,1287.8)(6,1558.8)(7,1836.3)(8,2119.0)(9,2403.5)(10,2690.2)
    };
    
\legend{Standard approach, Composite approach}    
\end{axis}
 \end{tikzpicture}
 \caption{Time needed to compute first-degree prime ideals of norm up to $M$ for a degree-$6$ defining polynomial.}
 \label{fig:results6}
\end{figure}
\FloatBarrier

As discussed in Section \ref{sec:Complexity}, the computational time appears to grow linearly with $M$, and the composite approach proves to be faster by a factor $\sim 1.5$.

\subsection{Extensions of smooth degrees} \label{sec:smooth}

According to the complexity estimations of Section \ref{sec:Complexity}, the composite approach is expected to be notably faster whenever the degree of the considered extensions has small prime factors.
Here we test an instance of such extensions with a moderately small extension.

We consider number fields of degree $315 = 3^2 \cdot 5 \cdot 7$, which can be obtained from their linearly disjoint number sub-fields of small degrees, as in the diagram below.

\FloatBarrier
\begin{figure}[!ht]
\centering
\begin{tikzpicture}[scale=0.8, transform shape]
\node (Q)                  {$\mathbb{Q}$};
\node[above left=1.7cm and 4cm of Q] (Q1)  {$\QQ(\a_1)$};
\node[right=2cm of Q1] (Q2)  {$\QQ(\a_2)$};
\node[above right=1.7cm and 4 cm of Q] (Q4) {$\QQ(\a_4)$}; 
\node[left=2 of Q4] (Q3)  {$\QQ(\a_3)$};
\node[above=4cm of Q] (Q1234) {$\QQ(\t) = \QQ(\a_1+\a_2+\a_3+\a_4)$};

\draw (Q) -- (Q1) node [midway, below, xshift = -0.35cm, yshift = 0.1cm] (TextNode) { $3$ } -- (Q1234) node [midway, below, xshift = -0.6cm, yshift = 0.3cm] (TextNode) { $105$ };
\draw (Q) --(Q2) node [midway, below, xshift = -0.25cm, yshift = 0.1cm] (TextNode) { $3$ } -- (Q1234) node [midway, below, xshift = -0.55cm, yshift = 0.1cm] (TextNode) { $105$ };
\draw (Q) -- (Q3) node [midway, below, xshift = 0.25cm, yshift = 0.1cm] (TextNode) { $5$ }  -- (Q1234)node [midway, below, xshift = 0.45cm, yshift = 0.1cm] (TextNode) { $63$ };
\draw (Q) -- (Q4) node [midway, below, xshift = 0.35cm, yshift = 0.1cm] (TextNode) { $7$ } -- (Q1234) node [midway, below, xshift = 0.45cm, yshift = 0.3cm] (TextNode) { $45$ };
\end{tikzpicture}
\caption{Lattice of the minimal fields in a number field of degree $315$. The large extension is realized as the compositum of the small underlying fields.}
\label{fig:lattice}
\end{figure}
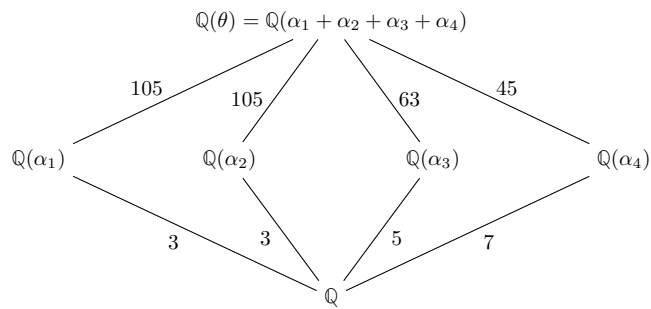
\FloatBarrier

A repeated application of Proposition \ref{idealsdownup} shows that we can compute the first-degree primes of $\ZZ[\t]$ by simply composing those of each $\ZZ[\a_i]$.
The time improvement with respect to the standard approach is noteworthy, as it is witnessed by Figure \ref{fig:speedComparison}.
In this case, the composite approach is $\sim 39$ times faster than the standard one.

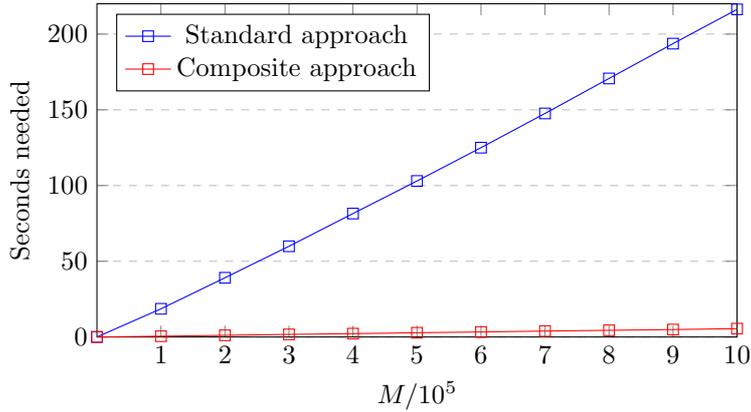
\begin{figure}[!htb] \centering
 \begin{tikzpicture}
\begin{axis}[ width=10cm,height=6cm,
    xlabel={$M/10^5$},
    ylabel={Seconds needed},
    xmin=0, xmax=10, 
    ymin=0, ymax=220, 
    xtick={1,2,3,4,5,6,7,8,9,10}, 
    ytick={0,50,100,150,200}, 
    legend pos=north west,
    ymajorgrids=true,
    grid style=dashed,
]

\addplot[
    color=blue,
    mark=square,
    ]
    coordinates { 
    (0,0)(1,18.6)(2,39.1)(3,59.9)(4,81.5)(5,103.0)(6,125.0)(7,147.6)(8,170.7)(9,193.7)(10,216.3)
    };

\addplot[
    color=red,
    mark=square,
    ]
    coordinates { 
    (0,0)(1,0.5)(2,1.1)(3,1.7)(4,2.2)(5,2.8)(6,3.3)(7,3.9)(8,4.4)(9,4.9)(10,5.5)
    };
    
\legend{Standard approach, Composite approach}    
\end{axis}
 \end{tikzpicture}
 \caption{Time needed to compute first-degree prime ideals of norm up to $M$ for a degree-$315$ defining polynomial.}
 \label{fig:speedComparison}
\end{figure}

In this setting, neither the degrees of the sub-fields are coprime nor the considered extensions are normal, so we should expect to miss a few first-degree primes.
We have considered ten randomly generated degree-$315$ number fields and we have collected the number of ideals constructed with the two approaches in Figure \ref{fig:diffNumberFDPI}.

\FloatBarrier
 \begin{table}[!htb] \centering
\begin{tabular}{ |c|c|c|c|c|c|c|c|c|c|c| } 
\cline{2-11}
 \multicolumn{1}{c|}{} & \multicolumn{10}{c|}{ $p$ ranging from $i \cdot 10^7$ to $(i+1) \cdot 10^7$} \\ 
\cline{1-11} \rule{0pt}{0.35cm}
 $10^7 \, \cdot$ & $i = 0$ & $i = 1$ & $i = 2$ & $i = 3$ & $i = 4$ & $i = 5$ & $i = 6$ & $i = 7$ & $i = 8$ & $i = 9$ \\
 \hline 
 Standard & 94759 & 83520 & 80137 & 79167 & 76478 & 74732 & 71694 & 75699 & 73324 & 72671 \\ 
 Composite & 94679 & 83518 &  80131 &  79166 & 76478 & 74732 & 71694 & 75698 & 73324 & 72671 \\
 \hline
 Difference & 80 & 2 & 6 & 1 & 0 & 0 & 0 & 1 & 0 & 0 \\
 \hline
\end{tabular}
 \caption{Number of norm-$p$ first-degree prime ideals constructed with the different approaches.}
 \label{fig:diffNumberFDPI}
 \end{table}
 \FloatBarrier

 In the instance portrayed by Table \ref{fig:diffNumberFDPI} the number of ideals that the composite approach misses in the general case is irrelevant, especially when their norm increases.
 This had to be expected from Proposition \ref{updown1}, as explained in Remark \ref{rmk:AlmostAlways}.

 
\section{Conclusions} \label{sec:conclusions}
 We have analyzed the behavior of first-degree prime ideals in composite extensions of number fields in terms of those arising from the underlying extensions, and we have characterized the cases when such correspondence is completely achieved. 
 Moreover, we have studied the divisibility of special-shaped principal ideals in every compositum of linearly disjoint fields in terms of the first-degree prime ideals of the underlying fields dividing the relative norms of the considered ideal. 
  
 Our work shows that the information of first-degree prime ideals of composite extensions may often be read from the underlying fields efficiently.
 Thus, when designing algorithms that deal with first-degree prime ideals, one may conceivably work inside small and easy-to-handle fields to achieve results in arbitrarily complex extensions. 
 In fact, we demonstrated that it is often sufficient and worthwhile to know the behavior of such prime ideals inside prime-degree number fields.

 A further investigation in this direction would require a deep and detailed study of the characteristics of the polynomials that are optimal for such an algorithm.
 In particular, the properties defined in \cite{Briggs} should be explored for the polynomials constructed as resultants, but such an analysis is solely focused on the application of this theory to the GNFS and goes beyond the scope of this paper.
 Moreover, the fast production of first-degree prime ideals is not the bottleneck of the state-of-the-art implementations of the GNFS, and the current heuristics suggest generators whose minimal polynomials have small coefficients.
 However, other types of compositions may be investigated to mimic the additive linear combination proposed in this work, in order to connect the properties of composite fields with those of their underlying subfields.
    
 
 \section*{Acknowledgments}
 The authors would like to thank professors Massimiliano Sala, Michele Elia, and Willem A. de Graaf for their useful advice and discussions. \\
 DT was supported in part by the European Union’s H2020 Programme, grant number ERC-669891, and in part by the Research Foundation - Flanders (FWO), project 12ZZC23N.
 

\end{document}